\documentclass[a4paper,11pt,intlimits,oneside]{amsart}
\usepackage{amsfonts, amssymb, amsmath}
\usepackage{enumerate}
\usepackage{esint}
\usepackage{latexsym,amssymb}

\newtheorem{thm}{Theorem}[section]

\newtheorem{lem}[thm]{Lemma}
\newtheorem{prop}[thm]{Proposition}
\newtheorem{defn}[thm]{Definition}

\newcommand{\lesi}{\lesssim}

\newcommand{\f}{\frac}

\newcommand{\vc}{\infty}
\newcommand{\dx}{d\mu(x)}
\newcommand{\dy}{d\mu(y)}

\newcommand{\tm}{t^{1/m}}
\newcommand{\sm}{s^{1/m}}

\newcommand{\BMOA}{{\rm BMO}_{\mathcal{A}}(X,w)}

\newcommand{\RR}{X}
\newcommand{\su}{\subset}
\newcounter{rek}
\title[Weighted BMO spaces associated to operators\footnotetext{}]{Weighted BMO spaces associated to operators}         % Enter your title between curly braces

\author{The Anh Bui}
\address{Department of Mathematics, Macquarie University, NSW 2109,
Australia}
\address{Department of Mathematics, University of Pedagogy, HoChiMinh City,
Vietnam}
\email{the.bui@mq.edu.au, bt\_anh80@yahoo.com}
\author{Xuan Thinh Duong}
\address{Department of Mathematics, Macquarie University, NSW 2109,
Australia}

\email{xuan.duong@mq.edu.au}

\keywords{weighted estimates, BMO spaces, Muckenhoupt weights, singular integrals,
  heat semigroups}
\subjclass[2010]{42B20, 42B35,
47B38}

\begin{document}

\date{}

\maketitle
\begin{abstract}
Let $X$ be a metric space equipped  with a metric $d$ and a nonnegative Borel measure $\mu$ satisfying the doubling property and let $\{\mathcal{A}_t\}_{t>0}$,
be a generalized approximations to the identity, for example  $\{\mathcal{A}_t\}$ is
a holomorphic semigroup $e^{-tL}$ with Gaussian upper bounds generated by an
operators $L$ on $L^2(X)$.
 In this paper, we introduce and study the weighted BMO space BMO$_\mathcal{A}(X,w)$ associated to the
 the family  $\{\mathcal{A}_t\}$.
 We show that for these spaces, the weighted John-Nirenberg inequality holds and we establish an
 interpolation theorem in scale of weighted $L^p$ spaces. As applications, we prove the
 boundedness of two singular integrals with non-smooth kernels on
 the weighted BMO space BMO$_\mathcal{A}(X,w)$.
\end{abstract}
\tableofcontents

\section{Introduction}

The introduction and development of the BMO (bounded mean oscillation)
function spaces on Euclidean spaces  in the 1960s played an important role in
modern harmonic analysis \cite{JN,CW}. The concept of spaces of homogeneous type,
which is a natural setting for the Calder\'on-Zygmund theory of singular integrals,
was introduced in the 1970s \cite{CW}. According to \cite{JN}, a locally integrable function $f$ defined on $\mathbb{R}^n$ is said
to be in ${\rm BMO}(\mathbb{R}^n)$, the space of functions of bounded mean oscillation, if
$$
\|f\|_{{\rm BMO}}=\sup_{B}\f{1}{|B|}\int_B|f(y)-f_B|dy<\vc
$$
where the supremum is taken over all balls $B$ in $\mathbb{R}^n$,
and $f_B$ stands for the mean of $f$ over $B$, i.e.,
$$
f_B=\f{1}{|B|}\int_Bf(y)dy.
$$

In \cite{FS}, Fefferman and Stein introduced the Hardy space $H^1(\mathbb{R}^n)$
and showed that the space ${\rm BMO}(\mathbb{R}^n)$ is the dual space of the
Hardy space $H^1(\mathbb{R}^n)$. They also proved the characterization of functions in
the BMO space by the Carleson measure.
In the study of boundedness of Calder\'on-Zygmund operators,
the Hardy space $H^1$ is a natural substitution to $L^1$ and the space BMO is a natural replacement for $L^\vc$.
Indeed, it is
well known that  Calder\'on-Zygmund operators are bounded on $L^p$ for $1<p<\vc$ but bounded neither on $L^1$ nor on $L^\vc$.
Meanwhile,
 Calder\'on-Zygmund operators map continuously from $H^1$ to $L^1$ and from $L^\vc$ to ${\rm BMO}$.
Moreover, one can obtain an interpolation theorem which gives $L^p$ boundedness from the boundedness on Hardy and on BMO spaces, i.e.
if a linear operator $T$ is bounded from $H^1$ to $L^1$ and bounded from $L^\vc$ to ${\rm BMO}$ then
by interpolation $T$ is bounded on $L^p$ for all $1<p<\vc$.

\medskip

In practical, there are large classes of operators whose kernels are not sufficiently smooth for them to belong to
the class of Calder\'on-Zygmund operators. It is possible that
certain operators are only bounded
on $L^p$ with the range of $p$ being a proper subset of $(1, \infty)$. In these cases,
the classical Hardy space and/or the classical BMO space are no longer suitable spaces for the study of boundedness of these operators. In this paper, given a family of operators $\{\mathcal{A}_t\}_{t>0}$ which is
a generalized approximations to the identity  (See its definition in Section 2)
and a suitable weight $w$, we develop the
theory of weighted BMO space ${\rm BMO}_{\mathcal{A}}(X,w)$ associated to $\mathcal{A}_t$. An important
example of the family $\{\mathcal{A}_t\}$ is when $\{\mathcal{A}_t\} = I - (I - e^{-tL})^M$ for some positive integer $M$
in which  $e^{-tL}$ is a holomorphic semigroup
 generated by an operator $L$ on $L^2(X)$, assuming that
 $L$ satisfies  Gaussian heat kernel upper bounds
and has a bounded $L^2$ holomorphic functional calculus.
The  new results in this article are the following:
\begin{enumerate}[(i)]
\item The introduction of weighted BMO space associated to generalized approximations of identity $\BMOA$ (Section 3.1);
\item The weighted John-Nirenberg inequality (Lemma \ref{lem-JNinequality}),
and the equivalence of ${\rm BMO}_\mathcal{A}^p(X,w)$ for $1\leq p<\vc$ (Theorem \ref{BMO^pequivalent});
\item An interpolation theorem concerning $\BMOA$ (Theorem \ref{interpolation1});
\item Applications to some singular integrals with non-smooth kernels (Section 5).
\end{enumerate}

We note that under suitable conditions on the operator $L$ and the weight $w$, the dual space of the weighted Hardy space $H^1_{L}(X,w)$ associated to the operator $L$ introduced in \cite{SY} (see also \cite{YY}) should be the weighted BMO spaces BMO$_L^*(X,w)$ in this paper. However, we  do not try to address this issue in this article.

Throughout the paper, we shall write $A\lesi B$ if there is a universal constant $C$ so that $A\leq CB$. Likewise, we shall write $A\sim B$ if $A\lesi B$ and $B\lesi A$.

The authors would like to thank Shaoxiong Hou for his useful comments on the first version of the paper.

\section{Preliminaries}

Let $X$ be a metric space equipped  with a metric $d$ and a nonnegative Borel measure $\mu$ satisfying the doubling property
$$
\mu(B(x, 2r )) \leq  C \mu(B(x, r )) < \vc
$$
for any $x \in X$ and $r > 0$, where the constant $C \geq 1$ is
independent of $x$ and $r$ and $B(x,r):=\{y: d(x,y)<r\}$.

Note that the doubling property implies the following strong
homogeneity property:
\begin{equation}\label{doublingproperty1}
\mu(B(x, \lambda r ))\leq  c\lambda^n \mu(B(x, r ))
\end{equation}
for some $c, n > 0$ uniformly for all $\lambda \geq 1$ and $x \in
X$. The value of parameter $n$ is a measure of the dimension of the space.
There also exist $c$ and $N, 0 \leq N \leq n$, so that
\begin{equation}\label{doublingproperty2}
\mu(B(y, r ))\leq  c\Big(1+\f{d(x,y)}{r}\Big)^N\mu(B(x, r ))
\end{equation}
uniformly for all $x, y \in X$ and $r > 0$. Indeed, property
$(\ref{doublingproperty2})$ with $N = n$ is a direct consequence of
the triangle inequality of the metric d and the strong homogeneity
property. In the cases of Euclidean spaces $\mathbb{R}^n$ and Lie
groups of polynomial growth, $N$ can be chosen to be $0$.\\

To simplify notation, for a measurable subset $E$ in $X$, we write $V(E)$ instead of $\mu(E)$. We will often use $B$ for $B(x_B, r_B)$.
Also given $\lambda > 0$, we will write $\lambda B$ for the
$\lambda$-dilated ball, which is the ball with the same center as
$B$ and with radius $r_{\lambda B} = \lambda r_B$ and denote $V(x,r)=\mu(B(x,r))$ for all $x\in X$ and $r>0$. For each ball $B\subset X$ we set
$$
S_0(B)=B \ \text{and} \ S_j(B) = 2^jB\backslash 2^{j-1}B \
\text{for} \ j\in \mathbb{N}.
$$
Recall that the Hardy-Littlewood maximal operator $M$ is defined by
$$
Mf(x)=\sup_{B\ni x}\f{1}{V(B)}\int_{B} | f(x) | \dx.
$$

We now give a simple covering lemma which states that we can cover a
given ball by a finite overlapping family of balls with smaller radii. This will be used frequently in the sequel.

\begin{lem}\label{coveringlemm}
For any ball $B(x_B, lr)$ in $X$, with $l\geq 1$ and $r>0$, then there exists a corresponding
family of balls $\{B(x_{1}, r),\ldots, B(x_{k}, r)\}$ such
that
\begin{enumerate}[(a)]
\item $B(x_{j}, r/3)\subset B(x_B, lr) $, for all $j=1,\ldots, k$;
\item $B(x_B, lr) \subset \cup_{j=1}^{k}B(x_{j}, r)$;
\item $k\leq Cl^n$;
\item $\sum _{j=1}^{k}\chi_{B(x_{j}, r)}\leq C$, where $C$ is
independent of $l$ and $r$.
\end{enumerate}
\end{lem}
The proof of this lemma is just a consequence of Vitali covering lemma and doubling property (\ref{doublingproperty1}). Hence we omit details here.

\subsection{Approximations to the identity}
We will work with a class of integral operators $\{\mathcal{A}_t\}_{t>0}$,
which plays the role of generalized approximations to the identity.
We assume that for each $t > 0$, the operator $\mathcal{A}_t$ is defined by
its kernel $a_t(x,y)$ in the sense that
$$
\mathcal{A}_tf(x)=\int_X a_t(x,y) f(y)\dy
$$
for every function $f\in\cup_{p\geq 1}L^p(X)$.\\
We also assume that the kernel $a_t(x,y)$ of $\mathcal{A}_t$ satisfies the
Gaussian upper bound
\begin{equation}\label{decaycondition}
|a_t(x,y)|\leq
\f{C}{V(x,\tm)}\exp\Big(-c\f{d(x,y)^{m/(m-1)}}{t^{1/(m-1)}}\Big),
\end{equation}
for all $t>0$ and $x,y \in X$ where $m$ is a positive constant, $m \ge 2$.

The decay of the kernel $a_t(x,y)$ guarantees that $\mathcal{A}_t$ is bounded
on $L^p(X)$ for all $p\in (1,\vc)$. More precisely, we have the
following proposition, see \cite{DR}.
\begin{prop}
For each $p \in [1, \vc]$, we have
$$
|\mathcal{A}_t f(x)|\lesi  Mf(x)
$$
for all $t> 0$ and $f\in L^p(X)$, $\mu$-a.e.
\end{prop}
\subsection{Muckenhoupt weights}

Throughout this article, we shall denote $w(E) :=\int_E w(x)\dx$ and $V(E)=\mu(E)$ for
any measurable set $E \subset X$. For $1 \leq p \leq \infty$ let $p'$
be the conjugate exponent of $p$, i.e. $1/p + 1/p' = 1$.

We first introduce some notation. We use the notation
$$
\fint_B h(x)\dx=\f{1}{V(B)}\int_Bh(x)\dx.
$$
A weight $w$ is a non-negative measurable and locally integrable function on $X$.
We say that $w \in A_p$, $1 < p < \infty$, if there exists a
constant $C$ such that for every ball $B \subset X$,
$$
\Big(\fint_B w(x)\dx\Big)\Big(\fint_B w^{-1/(p-1)}(x)\dx\Big)^{p-1}\leq C.
$$
For $p = 1$, we say that $w \in A_1$ if there is a constant $C$ such
that for every ball $B \subset X$,
$$
\fint_B w(y)\dy \leq Cw(x) \ \text{for a.e. $x\in B$}.
$$
We set $A_\vc=\cup_{p\geq 1}A_p$.\\

The reverse H\"older classes are defined in the following way: $w
\in RH_q, 1 < q < \infty$, if there is a constant $C$ such that for
any ball $B \subset X$,
$$
\Big(\fint_B w^q(y) \dy\Big)^{1/q} \leq C \fint_B w(x)\dx.
$$
The endpoint $q = \infty$ is given by the condition: $w \in
RH_\infty$ whenever, there is a constant $C$ such that for any ball
$B \subset X$,
$$
w(x)\leq C \fint_B w(y)\dy  \ \text{for a.e. $x\in B$}.
$$
Let $w \in A_\vc$, for $1\leq p <\infty$, the weighted spaces $L^p_w(X)$
can be defined by
$$\Big\{f :\int_{X} |f(x)|^p w(x)\dx < \infty\Big\}$$
with the norm
$$\|f\|_{L^p_w(X)}=\Big(\int_{X} |f(x)|^p w(x)\dx\Big)^{1/p}.$$

We sum up some of the standard properties of classes of weights  in the following lemma. For the proofs, see for example \cite{ST}.
\begin{lem}\label{weightedlemma1}
The following properties hold:
\begin{enumerate}[(i)]
\item $A_1\subset A_p\subset A_q$ for $1< p\leq q< \infty$.
\item $RH_\infty \su RH_q \su RH_p$ for $1< p\leq q< \infty$.
\item If $w \in A_p, 1 < p < \vc$, then there exists $1<r < p < \vc$ such that $w \in A_r$.
\item If $w \in RH_q, 1 < q < \vc$, then there exists $q < p < \vc$ such that $w \in RH_p$.
\item $ A_\vc =\cup_{1\leq p<\vc}A_p\subset  \cup_{1< q\leq \vc}RH_q$
\end{enumerate}
\end{lem}
\begin{lem}\label{weightedlemma2}
For any ball $B$, any measurable subset $E$ of $B$ and $w \in A_p, p \geq  1$,
 there exists a constant $C_1 > 0$ such that
$$
C_1\Big(\f{V(E)}{V(B)}\Big)^p\leq \f{w(E)}{w(B)}.
$$
If $w \in RH_r, r > 1$. Then, there exists a constant $C_2 > 0$ such that
$$
\f{w(E)}{w(B)}\leq C_2\Big(\f{V(E)}{V(B)}\Big)^{\f{r-1}{r}}.
$$
\end{lem}
From the first inequality of Lemma \ref{weightedlemma2}, if $w\in
A_1$ then there exists a constant $C>0$ so that for any ball $B\subset X$ and $\lambda >1$, we have
$$
w(\lambda B)\leq C\lambda^n w(B).
$$

\section{Weighted BMO spaces associated to operators}
\subsection{Definition of $\BMOA$}

Throughout this paper, we assume that the family of the operators $\{\mathcal{A}_t\}_{t\geq 0}$
satisfies the Gaussian upper bounds (\ref{decaycondition}) and
these operators commute, i.e. $\mathcal{A}_s\mathcal{A}_t=\mathcal{A}_t\mathcal{A}_s$ for all $s, t > 0$.
Note that we do not assume the semigroup property $\mathcal{A}_s\mathcal{A}_t = \mathcal{A}_{s+t}$
on the family $\{\mathcal{A}_t\}_{t\geq 0}$.\\

Following \cite{DY}, we now define the class of functions that
the operators $\{\mathcal{A}_t\}_{t\geq 0}$ act upon. A function $f\in
L^1_{loc}(X)$ is said to be a function of type $(x_0,\beta)$ if $f$
satisfies
\begin{equation}\label{rangeofP_t1}
\Big(\int_X\f{|f(x)|^2}{(1+d(x_0,x))^{\beta}V(x_0,1+d(x_0,x))}\dx\Big)^{1/2}
\leq c <\vc.
\end{equation}
We denote $M_{(x_0,\beta)}$ the collection of all functions of type
$(x_0,\beta)$. If $f\in M_{(x_0,\beta)}$, the norm of $f$ is defined
by
$$
||f||_{M_{x_0,\beta}}=\inf\{c: (\ref{rangeofP_t1}) \ \text{holds}\}.
$$
For a fixed $x_0\in X$, one can check that $M_{(x_0,\beta)}$  is
a Banach space under the norm $||f||_{M_{x_0,\beta}}$. For any
$x_1\in X$, $M_{(x_0,\beta)}=M_{(x_1,\beta)}$ with equivalent norms.
Denote by
$$
\mathcal{M}= \cup_{x_0\in X}\cup_{0<\beta<\vc}M_{(x_0,\beta)}.
$$

\begin{defn}
A function $f\in \mathcal{M}$ is said to be in $BMO_\mathcal{A}(w)$ with
$w\in A_\vc$, the space of functions of bounded mean oscillation
associated to $\{\mathcal{A}_t\}_{t\geq 0}$ and $w$, if there exists some
constant $c$ such that for any ball $B$,
\begin{equation}\label{BMOdef}
\f{1}{w(B)}\int_B|(I-\mathcal{A}_{t_B})f(x)|\dx\leq c,
\end{equation}
where $t_B= r_B^m$ ($m$ is a constant in (\ref{decaycondition})) and $r_B$ is the radius of $B$.\\
The smallest bound $c$ for which (\ref{BMOdef}) is satisfied is then
taken to be the norm of $f$ in this space and is denoted by
$||f||_{BMO_\mathcal{A}(X,w)}$.
\end{defn}

\textbf{Remark:} The space $(\BMOA, \|\cdot\|_{\BMOA})$ is a seminormed vector space, with
the seminorm vanishing on the space $\mathcal{K}_\mathcal{A}$, defined by
$$
\mathcal{K}_\mathcal{A} =\{f\in \mathcal{M}: \mathcal{A}_tf(x)=f(x) \ \text{for almost all $x$ and for all $t>0$}\}.
$$
In this paper, $\BMOA$ space is understood to be modulo $\mathcal{K}_A$.\\

The following result gives a sufficient condition for the BMO$(X,w)$ to be contained in $\BMOA$.
The proof for the unweighted case was given in \cite{Ma} (see also \cite{DY}).

\begin{prop}\label{BMOA and BMO}
Suppose that $w\in A_1$  and $\mathcal{A}_t(1)=1$ for all $t>0$, i.e., $\int_Xa_t(x,y)\dy=1$ for almost all $x\in X$. Then the
inclusion ${\rm BMO}(X,w)\subset \BMOA$ holds where
$$
{\rm BMO}(X,w):=\{f\in L^1_{{\rm loc}}: \|f\|_{{\rm BMO}(X,w)}:=\sup_{B}\f{1}{w(B)}\int_B|f-f_B|d\mu<\vc\}.
$$
\end{prop}
\begin{proof} Let $f\in {\rm BMO}(X,w)$. For any ball $B$, due to $\mathcal{A}_t(1)=1$, we have
\begin{equation*}
\begin{aligned}
\f{1}{w(B)}\int_B&|f(x)-\mathcal{A}_{t_B}f(x)|\dx\\
&=\f{1}{w(B)}\int_B\Big|f(x)-\int_X a_{t_B}(x,y)f(y)\dy\Big|\dx\\
&=\f{1}{w(B)}\int_B\Big|\int_X a_{t_B}(x,y)(f(x)-f(y))\dy\Big|\dx\\
&=\f{1}{w(B)}\int_B\int_X \Big|a_{t_B}(x,y)(f(x)-f(y))\Big|\dy\dx\\
&\leq \f{C}{V(B)w(B)}\int_B\int_X \Big|\exp\Big(-c\f{d(x,y)^{m/(m-1)}}{t_B^{1/(m-1)}}\Big)(f(x)-f(y))\Big|\dy\dx\\
&= \f{C}{V(B)w(B)}\int_B\int_{2B} \ldots \dy\dx+\sum_{j\geq 2}\f{1}{V(B)w(B)}\int_B\int_{S_j(B)} \ldots \dy\dx \\
&=I+\sum_{j\geq 2}I_j.
\end{aligned}
\end{equation*}
Let us estimate $I$ first. We have
\begin{equation*}
\begin{aligned}
I&\lesi \f{1}{V(B)w(B)}\int_B\int_{2B} |f(x)-f_B|\dy\dx+\f{1}{V(B)w(B)}\int_B\int_{2B} |f(y)-f_{2B}|\dy\dx\\
&~~~+\f{1}{V(B)w(B)}\int_B\int_{2B} |f_{2B}-f_{B}|\dy\dx\\
&\lesi \|f\|_{{\rm BMO}(X,w)}.
\end{aligned}
\end{equation*}
For the term $I_j, j\geq 2$, we have
\begin{equation*}
\begin{aligned}
I_j&\lesi \f{1}{V(B)w(B)}\int_B\int_{2^jB} \Big|\exp(-c2^{jm/(m-1)})(f(x)-f(y))\Big|\dy\dx\\
&\lesi \f{\exp(-c2^{jm/(m-1)})}{V(B)w(B)}\Big(\int_B\int_{2^jB} |f(x)-f_B|\dy\dx+\int_B\int_{2^jB} |f(y)-f_{2^jB}|\dy\dx\\
&~~~~~~~+\int_B\int_{2^jB} |f_{2^jB}-f_B|\dy\dx\Big)\\
&\lesi \|f\|_{{\rm BMO}(X,w)}.
\end{aligned}
\end{equation*}
These estimates on $I$ and $I_j, j\geq 2$ give $\|f\|_{\BMOA}\leq \|f\|_{{\rm BMO}(X,w)}$. This completes our proof.
\end{proof}

\begin{prop}\label{pro1}
For $t>0, K>1$ and $w\in A_1$ we have for a.e. $x\in X$
$$|(\mathcal{A}_{t}f(x)- \mathcal{A}_{Kt}f(x))|\lesi (1+\log K)\f{w(B(x,\tm))}{V(x,\tm)} \|f\|_{\BMOA}.$$
\end{prop}

Before coming to the proof, we would like to mention that the same estimates as in Proposition \ref{pro1} was obtained in \cite{DY}
under the extra assumption of semigroup property on the family $\{\mathcal{A}_t\}$. While the argument in \cite{DY} relies on Christ's covering lemma,
our argument uses Lemma \ref{coveringlemm}.

\begin{proof} For any $s, t>0$ such that $t\leq s \leq 2t$, we
have
$$
|\mathcal{A}_{t}f(x)- \mathcal{A}_{s}f(x)|\leq |\mathcal{A}_t((I- \mathcal{A}_s)f(x))|+|\mathcal{A}_s((I-\mathcal{A}_t)f(x))|:=I_1+I_2.
$$
We first estimate $I_1$. The Gaussian upper bound (\ref{decaycondition}) of $\mathcal{A}_t$ and the fact that $t\approx s$ gives that
\begin{equation*}
\begin{aligned}
I_1 &\lesi \f{1}{V(x,\tm)}\int_X \exp\Big(-c\f{d(x,y)^{m/(m-1)}}{t^{1/(m-1)}}\Big)|(I- \mathcal{A}_s)f(y)|\dy\\
&\lesi \f{1}{V(x,\sm)}\int_{B(x,\sm)}
\exp\Big(-c\f{d(x,y)^{m/(m-1)}}{s^{1/(m-1)}}\Big)|(I- \mathcal{A}_s)f(y)|\dy\\
&~~~~ +\sum_{j\geq 2}\f{1}{V(x,\sm)}\int_{S_j(B(x,\sm))}
\exp\Big(-c\f{d(x,y)^{m/(m-1)}}{s^{1/(m-1)}}\Big)|(I- \mathcal{A}_s)f(y)|\\
&=I_{11}+\sum_{j\geq 2}I_{1j}.
\end{aligned}
\end{equation*}
For the term $I_{11}$, since $t\approx s$ and $w\in A_1$, we have
\begin{equation*}
\begin{aligned}
I_{11}&\leq \|f\|_{\BMOA}\f{w(B(x,\sm))}{V(x,\sm)}\leq
C\|f\|_{\BMOA}\f{w(B(x,\tm))}{V(x,\tm)}.
\end{aligned}
\end{equation*}

For $j\geq 2$, using Lemma \ref{coveringlemm} we can cover the
annulus $S_j(B(x,\sm))$ by a finite overlapping family of at most
$C2^{jn}$ balls $B(x_k^j, \sm)$. Using $w\in A_1$, we can dominate the term
$I_{1j}$ as follows.
\begin{equation*}
\begin{aligned}
I_{1j}&\lesi \f{1}{V(x,\sm)}\int_{S_j(B(x,\sm))}
\exp\Big(-c\f{d(x,y)^{m/(m-1)}}{s^{1/(m-1)}}\Big)|(I- \mathcal{A}_s)f(y)|\dy\\
&\lesi \f{1}{V(x,\sm)}\int_{S_j(B(x,\sm))}
e^{-c2^{j/(m-1)}}|(I- \mathcal{A}_s)f(y)|\dy\\
&\lesi \sum_{k}\f{1}{V(x,\sm)}\int_{B(x^j_k,\sm)}
e^{-c2^{j/(m-1)}}|(I- \mathcal{A}_s)f(y)|\dy\\
&\lesi \sum_{k}\f{w(B(x^j_k,\sm))}{V(x,\sm)}e^{-c2^{j/(m-1)}}\|f\|_{\BMOA}\\
&\lesi \f{w(B(x,2^j\sm))}{V(x,\sm)}e^{-c2^{j/(m-1)}}\|f\|_{\BMOA}\\
&\lesi 2^{jn}\f{w(B(x,2^j\sm))}{V(x,s^{1/m})}e^{-c2^{j/(m-1)}}\|f\|_{\BMOA}\\
&\lesi 2^{jn}e^{-c2^{j/(m-1)}}\|f\|_{\BMOA}\f{w(x,t^{1/m})}{V(x,t^{1/m})}.
\end{aligned}
\end{equation*}
This implies
$$
I_1\leq C\|f\|_{\BMOA}\f{w(B(x,t^{1/m}))}{V(x,t^{1/m})}.
$$
A similar argument also gives
$$
I_2\leq C\|f\|_{\BMOA}\f{w(B(x,t^{1/m}))}{V(x,t^{1/m})}.
$$
Therefore, we have
\begin{equation}\label{eq1-proofBMOinequality}
|(\mathcal{A}_{t}f(x)-\mathcal{A}_{t+s}f(x))|\lesi \|f\|_{\BMOA}\f{w(B(x,t^{1/m}))}{V(x,t^{1/m})}
\end{equation}
for all $t\leq s\leq 2t$.

In general case,
taking $l\in \mathbb{N}$  such that $2^l\leq K<2^{l+1}$, we can
write
\begin{equation}\label{eq1-pro1}
\begin{aligned}
|(\mathcal{A}_{t}f(x)- \mathcal{A}_{Kt}f(x))|&\leq
\sum_{k=1}^{l}|\mathcal{A}_{2^{l-1}t}f(x)- \mathcal{A}_{2^lt}f(x)|+| \mathcal{A}_{2^lt}f(x)- \mathcal{A}_{Kt}f(x)|\\
&\lesi \sum_{k=1}^{l}\|f\|_{\BMOA}\f{w(B(x,2^{l-1}t^{1/m}))}{V(x,2^{l-1}t^{1/m})}.
\end{aligned}
\end{equation}
Since $w\in A_1$, we have
$$
\f{w(B(x,2^{k}t^{1/m}))}{V(x,2^{k}t^{1/m})}\leq C\f{w(B(x,t^{1/m}))}{V(x,t^{1/m})}
$$
for all $k\geq 0$.

This together with (\ref{eq1-pro1}) gives
$$
|(\mathcal{A}_{t}f(x)-\mathcal{A}_{Kt}f(x))|\lesi (1+\log K)\|f\|_{\BMOA}\f{w(B(x,t^{1/m}))}{V(x,t^{1/m})}.
$$
This completes the proof.
\end{proof}
\subsection{John-Nirenberg inequality on BMO$_\mathcal{A}(X,w)$}

In this section, we will show that functions in the new class of weighted BMO spaces BMO$_\mathcal{A}(X,w)$ satisfy
the John-Nirenberg inequality. The unweighted version was obtained in \cite{DY}.
Here, we extend to the weighted BMO spaces associated to the family of operators $\{\mathcal{A}_t\}_{t>0}$.

\begin{defn}
For $w\in A_1$ and $p\in [1,\vc),$ the function $f\in \mathcal{M}$
is said to be in ${\rm BMO}_\mathcal{A}^p(X,w)$, if there exists some constant $c$
such that for any ball $B$,
\begin{equation}\label{BMOdef1}
\Big(\f{1}{w(B)}\int_B|(I- \mathcal{A}_{t_B})f(x)|^pw^{1-p}(x)\dx\Big)^{1/p}\leq c.
\end{equation}
where $t_B = r_B^m$ and $r_B$ is the radius of $B$.\\
The smallest bound $c$ for which (\ref{BMOdef1}) holds is
then taken to be the norm of $f$ in this space and is denoted by
$||f||_{{\rm BMO}_\mathcal{A}^p(X,w)}$.
\end{defn}

Similar to the classical case, it turns out that the spaces  BMO$_{\mathcal{A}}^p(X,w)$ are equivalent for all $1\leq p < \vc$.
More precisely, we have the following result.

\begin{thm}\label{BMO^pequivalent}
For $w\in A_1$ and $p\in [1,\vc),$ the spaces ${\rm BMO}_\mathcal{A}^p(X,w)$ coincide and their
norms are equivalent.
\end{thm}

Before coming to the proof Theorem \ref{BMO^pequivalent} we need the following result.

\begin{thm}\label{lem-JNinequality}
For $w\in A_1$ and $f\in \BMOA$,  there exist positive constants
$c_1$ and $c_2$ such that for any ball $B$ and $\lambda>0$ we have
\begin{equation}\label{equ1-John-Nirenberg}
w\{x\in B:|(f(x)-A_{t_B}f(x))w^{-1}(x)|>\lambda\}\leq c_1
w(B)\exp\Big(-\f{c_2\lambda}{\|f\|_{\BMOA}}\Big).
\end{equation}
\end{thm}

\begin{proof} Let us recall that if $w\in A_\vc$, there
exist $C>0$ and $\delta>0$ such that for any ball $B$ and any
measurable subset $E\subset B$ we have
$$
\f{w(E)}{w(B)}\leq C\Big(\f{\mu(E)}{\mu(B)}\Big)^{\delta}.
$$
So, to prove (\ref{equ1-John-Nirenberg}), it suffices to show that
\begin{equation}\label{eq1-proofJN}
\mu\{x\in B:|(f(x)-\mathcal{A}_{t_B}f(x))w^{-1}(x)|>\lambda\}\leq c_1
\mu(B)\exp\Big(-\f{c_2\lambda}{\|f\|_{\BMOA}}\Big).
\end{equation}
The proof of (\ref{eq1-proofJN}) is similar to that of Theorem 3.1 in \cite{DY} in which Proposition 2.6 in \cite{DY} is replaced by Proposition \ref{pro1}. However, for reasons of completeness, we sketch out the proof here.

Without the loss of generality, we may assume that $\|f\|_{\BMOA}=1$. Then we need to claim that for all balls $B$, we have
\begin{equation}\label{eq2-proofJN}
\mu\{x\in B:|(f(x)-\mathcal{A}_{t_B}f(x))w^{-1}(x)|>\lambda\}\leq c_1
e^{-c_2\lambda}\mu(B).
\end{equation}
If $\alpha<1$, (\ref{eq2-proofJN}) holds for $c_1=e$ and $c_2=1$. Hence, we consider the case $\alpha\geq 1$.

For any ball $B\subset X$, we set
$$
f_0 = [(f(x)-\mathcal{A}_{t_B}f(x))w^{-1}(x)]\chi_{10B}.
$$
Then, using the fact that $w\in A_1$, we have
$$
\begin{aligned}
\|f_0\|_{L^1}&\leq \int_{10B}|(I-\mathcal{A}_{t_B})f(x)|w^{-1}(x)\dx\\
&\leq \f{V(10B)}{w(10B)}\int_{10B}|(I-\mathcal{A}_{t_B})f(x)|\dx\\
&\lesi V(B)\|f\|_{\BMOA} = V(B).
\end{aligned}
$$
Taking $\beta>1$, we set
$$
F=\{x: M(f_0)(x)\leq \beta\} \ \ \text{and} \ \ \Omega=X\setminus F.
$$
Then we can pick a family of balls $\{B_{1,i}\}_{i=1}^\vc$ so that
\begin{enumerate}[(i)]
\item $\cup_{i}B_{1,i}=\Omega$;
\item there exists $\kappa>0$ so that $\sum_{i}\chi_{B_{1,i}}\leq \kappa$;
\item there exists $c_0$ such that $c_0B_{1,i}\cap F\neq \emptyset$ for all $i$.
\end{enumerate}
See \cite[Chaptier]{CW}.

For $x\in B\setminus [\cup_{i}B_{1,i}]$, by (i), we have
$$
|(I-\mathcal{A}_{t_B})f(x)|w^{-1}(x)=|f_0(x)|\chi_{F}(x)\leq M(f_0)(x)\chi_{F}(x)\leq \beta.
$$
Moreover, from (ii)-(iii) and the fact that the Hardy-Littlewood maximal function $M$ is of weak type $(1,1)$, we have
$$
\begin{aligned}
\sum_{i}\mu(B_{1,i})&\lesi \mu(\Omega)\lesi \f{1}{\beta}\|f_0\|_{L^1}\\
&\leq\f{c_3}{\beta}V(B).
\end{aligned}
$$
By using argument as in \cite[pp. 24-25]{DY}, we can prove that for $B_{1,i}\cap B\neq \emptyset$, we have
$$
|(\mathcal{A}_{t_{B_{1,i}}}-\mathcal{A}_{t_B})f(x)|w^{-1}(x)\leq c_4\beta
$$
for all $x\in B_{1,i}$.

On each $B_{1,i}$, repeat the argument above with the function
$$
f_{1,i}=[(I-\mathcal{A}_{t_{B_{1,i}}})f(x)w^{-1}(x)]\chi_{10B_{1,i}}
$$
and the same value $\beta$. Then we can pick the family of balls $\{B_{2,m}\}_{m=1}^\vc$ such that
\begin{enumerate}[(a)]
\item for any $x\in B_{1,i}\setminus [\cup_{m}B_{2,m}]$, $|(I-\mathcal{A}_{t_{B_{1,i}}})f(x)|w^{-1}(x)\leq \beta$;
\item $\sum_m \mu(B_{2,m})\leq \f{c_3}{\beta}V(B_{1,i})$;
\item for any $B_{2,m}\cap B_{1,i}\neq \emptyset$, $|(\mathcal{A}_{t_{B_{2,m}}}-\mathcal{A}_{t_{B_{1,i}}})f(x)|w^{-1}(x)\leq c_4\beta$, for all $x\in B_{2,m}$.
\end{enumerate}
We now abuse  the notation $\{B_{2,m}\}$ for the family of all families $\{B_{2,m}\}$ corresponding to different $B_{1,i}'s$. Then, for all $x\in B\setminus [\cup_m B_{2,m}]$ we have
$$
|(I-\mathcal{A}_{t_{B}})f(x)|w^{-1}(x)\leq |(I-\mathcal{A}_{t_{B_{1,i}}})f(x)|w^{-1}(x)+|(\mathcal{A}_{t_{B_{1,i}}}-\mathcal{A}_{t_{B}})f(x)|w^{-1}(x)\leq 2c_4\beta
$$
and
$$
\sum_m \mu(B_{2,m})\leq \Big(\f{c_3}{\beta}\Big)^2V(B).
$$
In the consequence, for each $K\in \mathbb{N}_+$ we can find a family of balls $\{B_{K,m}\}_{m=1}^\vc$ so that
$$
|(I-\mathcal{A}_{t_{B}})f(x)|w^{-1}(x)\leq Kc_4\beta \ \ \text{for all} \ \ x\in B\setminus [\cup_m B_{k,m}]
$$
and
$$
\sum_m \mu(B_{K,m})\leq \Big(\f{c_3}{\beta}\Big)^KV(B).
$$
If $Kc_4\beta\leq \alpha\leq (K+1)c_4\beta$ for all $K\in \mathbb{N}_+$, we have
$$
\begin{aligned}
\mu\{x\in B:|(f(x)-\mathcal{A}_{t_B}f(x))|w^{-1}(x)>\lambda\}&\leq \sum_m \mu(B_{K,m})\leq \Big(\f{c_3}{\beta}\Big)^KV(B)\\
&\leq \sqrt{\beta} \exp\Big(-\f{\alpha\log\beta}{4c_4\beta}\Big)V(B)
\end{aligned}
$$
provided $\beta>c_3^2$.

If $\alpha<c_4\beta$, then
$$
\mu\{x\in B:|(f(x)-\mathcal{A}_{t_B}f(x))|w^{-1}(x)>\lambda\}\lesi e^{-\f{\alpha}{c_4\beta}}V(B).
$$
Hence, this completes our proof.
\end{proof}

\emph{Proof of Theorem \ref{BMO^pequivalent}:} For $f\in {\rm BMO}_\mathcal{A}^p(X,w)$, using H\"older's
inequality, we have, for all balls $B$,
\begin{equation}
\begin{aligned}
\f{1}{w(B)}\int_B&|(I- \mathcal{A}_{t_B})f(x)|\dx\\
&\leq \f{1}{w(B)}\Big(\int_B|(I-\mathcal{A}_{t_B})f(x)|^pw^{1-p}(x)\dx\Big)^{1/p}\Big(\f{1}{w(B)}\int_Bw(x)d\mu(x)\Big)^{1/p'}\\
&\leq \|f\|_{{\rm BMO}_\mathcal{A}^p(X,w)}.
\end{aligned}
\end{equation}
This implies that ${\rm BMO}_\mathcal{A}^p(X,w) \subset \BMOA$.

Conversely, by Lemma \ref{lem-JNinequality}, we have for any $f\in
\BMOA$, the ball $B$ and $p\in [1, \vc)$,
\begin{equation*}
\begin{aligned}
\f{1}{w(B)}&\int_B|(I-\mathcal{A}_{t_B})f(x)|^pw^{1-p}(x)\dx\\
&=\f{1}{w(B)}\int_B|(I-\mathcal{A}_{t_B})f(x)w^{-1}(x)|^pw(x)\dx\\
&=\f{p}{w(B)}\int_0^\vc \lambda^{p-1}w\{x\in B:
|(I-\mathcal{A}_{t_B})f(x)w^{-1}(x)|>\lambda\}d\lambda\\
&\leq c_p\f{1}{w(B)}\int_0^\vc \lambda^{p-1}w(B)\exp\Big(-c_2\f{\lambda}{\|f\|_{\BMOA}}\Big)d\lambda\\
&\leq c_p\|f\|_{\BMOA}^p.
\end{aligned}
\end{equation*}
The proof is complete.
\begin{flushright}
    $\Box$
\end{flushright}

\section{An Interpolation Theorem}

In this section, we study the interpolation of the weighted BMO
space $\BMOA$ in general setting of spaces of homogeneous type.
Firstly, We review the concept of the sharp maximal operator
$M_\mathcal{A}^\sharp $ associated to the family $\{\mathcal{A}_t\}_{t>0}$ defined on $L^p(X), p>0$ as well as its basic
properties in \cite{Ma},
$$
M_\mathcal{A}^\sharp f(x)=\sup_{x\in
B}\Big(\f{1}{\mu(B)}\int_B|(I- \mathcal{A}_{t_B})f(x)|\dx\Big),
$$
where $t_B=r^m_B$.\\
We recall the following results in \cite{Ma}.

\begin{thm}\label{Ma2} Let $0<p<\vc$ and $w\in A_\vc$. For every $f \in L^1_0(X)$ with $Mf \in L^p_w(X)$, we have
\begin{enumerate}[(i)]
\item $M^\sharp_\mathcal{A}f(x) \leq  CMf(x)$.
\item $||Mf||_{L^p_w(X)} \leq  C ||M_ \mathcal{A}^\sharp f||_{L^p_w(X)}$ if $\mu(X)=\vc$.
\item $||Mf||_{L^p_w(X)} \leq  C (||M_\mathcal{A}^\sharp f||_{L^p_w(X)}+||f||_{L^1})$ if $\mu(X)<\vc$.
\end{enumerate}
\end{thm}
In what follows, the operator $T$ is said to be bounded from
$wL^\vc$ to $BMO_\mathcal{A}(X,w)$ if there exists $c$ such that for all $f\in
L^\vc(X)$,
$$
\|T(fw)\|_{\BMOA}\lesi\|f\|_{L^\vc}.
$$

We recall an interpolation theorem for
the classical weighted BMO in \cite{B}.
\begin{thm}\label{thm-B}
Let $T$ be a linear operator which is bounded on $L^2(\mathbb{R}^n)$. Assume that $T$ and $T^*$ are bounded from $wL^\vc$ to ${\rm BMO}(X,w)$ for all $w\in A_1$. Then $T$ is bounded on $L^p_w(\mathbb{R}^n)$ for all $1<p<\vc$ and $w\in A_p$.
\end{thm}

It is interesting that our weighted
$\BMOA$ can be considered as a good substitution the classical
weighted BMO in the sense of interpolation. By adapting the arguments in \cite{B} to our situation, we will establish an interpolation theorem
concerning the our weighted BMO spaces $\BMOA$ which generalizes Theorem \ref{thm-B} to the range of weights and to the weighted BMO spaces associated to the family $\{\mathcal{A}_t\}_{t>0}$.

\begin{thm}\label{interpolation1}
Assume that $T$ is a linear operator which is bounded on $L^2(X)$. Assume also that $T$
is bounded from $wL^\vc$ to ${\rm BMO}_\mathcal{A}(X,w)$ and $T^*$ is bounded from $wL^\vc$ to
${\rm BMO}_{\mathcal{A}^*}(X,w)$ for all $w\in A_1\cap RH_s$ with $1\leq s<\vc$. Then $T$ is
bounded on $L^p_w(X)$ for all $s<p<\vc$, $w\in A_{p/s}$.
\end{thm}
\begin{proof} For the sake of simplicity we assume that
$\mu(X)=\vc$, the case that $\mu(X) < \infty$ can be treated in the same
way. When $w\equiv 1$, the operator $T$ is bounded from $L^\vc$ to
$\BMOA$. Due to \cite[Theorem 5.2]{DY}, $T$ is bounded on $L^p(X)$
for $p\in (2,\vc)$. By duality, one gets that $T$ is bounded on
$L^p(X)$ for $p\in (1,\vc)$.

Now for $w\in A_1\cap RH_s$ and $f\in L^\vc(X)$, we have
\begin{equation*}
\begin{aligned}
w^{-1}(x)M_\mathcal{A}^\sharp(T^*(wf))(x)&=\sup_{B\ni
x}\f{1}{V(B)}\int_B|(I- \mathcal{A}_{t_B})T^*(wf)(y)|\dy w^{-1}(x)\\
&\leq \sup_{B\ni
x}\f{1}{w(B)}\int_B|(I-\mathcal{A}_{t_B})T^*(wf)(y)|\dy \\
&\leq c\|T^*(wf)\|_{\BMOA}\leq c\|f\|_{L^\vc}
\end{aligned}
\end{equation*}
for all $x\in X$. This implies that the operator
$w^{-1}M^\sharp_{\mathcal{A},T^*w}$ is bounded on $L^\vc(X)$, where
$M^\sharp_{\mathcal{A},T^*w}$ is defined by $M^\sharp_{\mathcal{A},T^*w}f= M^\sharp_\mathcal{A} (T^*(wf))$.
On the other hand due to Proposition \ref{Ma2} and the
$L^2$-boundedness of $T^*$, $M^\sharp_{\mathcal{A},T^*}$ is bounded on
$L^2(X)$. This together with the interpolation, see for example \cite{BT}, gives
$$
u^{2/p-1} M^\sharp_\mathcal{A}(T^* u^{1-2/q}): L^q \rightarrow L^q,
$$
where $\f{1}{p}+\f{1}{q}=1$.%, $p$ near 1.

This implies
$$
M^\sharp_{\mathcal{A},T^*}: L^q(w^{2-q}) \rightarrow L^q(w^{2-q}).
$$
Using Theorem \ref{Ma2}, we have
$$
T^*: L^q(w^{2-q}) \rightarrow L^q(w^{2-q}).
$$
Let $g\in L^q(w^{2-q})$ and $f\in L^p(w^{2-p})$. We have
$$
\int_X|(Tf)g|d\mu=\int_X |fw^{1-2/q}
(T^*g)w^{2/q-1}|d\mu\leq\|T^*g\|_{L^q(w^{2-q})}\|f\|_{L^p(w^{2-p})}.
$$
By duality, $T: L^p(w^{2-p}) \rightarrow L^p(w^{2-p})$, or,
$w^{2/p-1}Tw^{1-2/p}: L^p\rightarrow L^p$.

On the other hand, for $f\in L^p$ and $g\in L^q$, we have
$$
\int_X |T(fw^{2/q -1})w^{1-2/q}g|d\mu=\int_X |f\times w^{2/q
-1}T^*(w^{1-2/q}g)|d\mu\leq c\|f\|_{L^p}\|g\|_{L^q},
$$
and hence $w^{1-2/q}Tw^{2/q-1}: L^p\rightarrow L^p$.

Since we can interchange $T$ and $T^*$, we can show that for
$\f{1}{p}+\f{1}{q}=1$, $p$ near 1, and $w, v\in A_1$,
$$
w^{1-2/q}Tw^{2/q-1}: L^p\rightarrow L^p \ \text{and} \
v^{2/q-1}Tv^{1-2/q}: L^q\rightarrow L^q.
$$
By interpolation, we obtain
$$
w^{\alpha(t)}v^{\beta(t)}T(w^{-\alpha(t)}v^{-\beta(t)}): L^{1/t}
\rightarrow L^{1/t} \ \text{for} \ \f{1}{q}\leq t\leq \f{1}{p}
$$
for all $v, w\in A_1\cap RH_s$, where $\alpha(t)=t-\f{1}{q}$ and $\beta(t)=t-\f{1}{p}$.

This gives $T:L^{p_0}(u) \rightarrow L^{p_0}(u)$ whenever
\begin{equation}\label{eq1-interthm}
u=w^{p_0\alpha(1/p_0)}v^{p_0\beta(1/p_0)}, \ \ w,v\in A_1\cap RH_s \ \text{and $p<p_0<q$}.
\end{equation}
Take $p_0=(q+s)-qs/p$ and $r_0=\f{pq}{q-p}$. For any $u\in A_{p_0/s}$, by Jones Factorization, there exist $u_1, u_2\in
A_1$ such that $u=u_1u_2^{1-p_0/s}$, see \cite{J}. Setting $u_1=w_1^s$ and $u_2=w_2^s$, then $w_1, w_2\in A_1\cap RH_s$.  Hence, we can pick $\delta>0$ so that $u_1^{1+\delta},u_2^{1+\delta}\in A_1$. For $p$ close enough to $1$, $r_0<1+\delta$ and hence $u_1^{r_0}=w_1^{r_0s}, u_2^{r_0}=w_2^{r_0s}\in A_1$. This implies $w_1^{r_0}, w_2^{r_0}\in A_1\cap RH_s$. Due to (\ref{eq1-interthm}), $T$ is bounded on $L^{p_0}(v)$, here $v=(w_1^{r_0})^{p_0\alpha(1/p_0)}(w_2^{r_0})^{p_0\beta(1/p_0)}=w^s_1w_2^{s(1-p_0)}=u$. Applying Theorem 4.9 in \cite{AM}, $T$  is bounded on $L^p_w(X)$ for all $s<p<\vc$ and $w\in A_{p/s}$. This completes our proof.
\end{proof}

\section{Applications to boundedness of singular integrals}

Let $X$ be a space of homogeneous type $(X,
d,\mu)$. Let $T$ be a bounded linear operator from $L^2(X)$ to $L^2(X)$ with kernel $k$ such
that for every $f$ in $L^\vc(X)$ with bounded support,
$$
Tf(x)=\int_{X}k(x,y)f(y)\dy,
$$
for $\mu$-almost all $x \notin suppf$. We will consider the following conditions:

\textbf{(H1)} There exists a class of approximation to the identity
$\{\mathcal{A}_t\}_{t>0}$ satisfying (\ref{decaycondition}) such that the
operators $(T - \mathcal{A}_tT)$ and $(T - T\mathcal{A}_t)$ have associated kernels
$K^1_t (x, y)$ and $K^2_t (x, y)$ respectively and there exist
positive constants $\alpha$ and $c_1, c_2$ such that
$$
\max \{|K^2_t(x,y)|, |K^1_t(x,y)| \} \leq
c_2\f{1}{V(x,d(x,y))}\f{t^{\alpha/m}}{d(x,y)^\alpha}
$$
when $d(x,y)\geq c_1t^{1/m}$.

\textbf{(H2)} There exists a class of approximation to the identity
$\{\mathcal{A}_t\}_{t>0}$ satisfying (\ref{decaycondition}) such that the
operators $(T - \mathcal{A}_tT)$ and $(T - T\mathcal{A}_t)$ have associated kernels
$K^1_t (x, y)$ and $K^2_t (x, y)$ so that there exist
$1 < p_0 < \vc$ and $\delta>0$ such that for any ball $B\subset X$ we have
\begin{equation}\label{eq1-maintheorem}
\Big(\int_{S_j(B)}|K^i_{r^m_B}(z,y)|^{p_0}\dy\Big)^{1/p_0}\lesi 2^{-j\delta}V(2^jB)^{1/p_0-1}
\end{equation}
for all $z\in B$, all $j\geq 2$ and $i=1,2$.

It was proven in \cite{DM} that if $T$ is an operator satisfying
(H1) or (H2) above, then $T$ bounded on $L^p(X)$ for $1 < p<2$. Note
that condition (H2) does not require the regularity assumption on
space variables. This allows us to obtain $L^p$-boundedness of
certain singular integrals with nonsmooth kernels such as the
holomorphic functional calculi and spectral multipliers of $L$, see Subsections 5.1 and 5.2.

We now prove the following
theorems:
\begin{thm}\label{thm-appl1}
Let $T$ be an operator satisfying (H1). Then for any $w\in
A_1$, $T$ and $T^*$ are bounded from $wL^\vc(X)$ to $\BMOA$ and from
$wL^\vc(X)$ to ${\rm BMO}_{\mathcal{A}^*}(X)$. Then, by interpolation, $T$ is
bounded on $L^p_w(X)$ for all $p\in (1,\vc)$ and $w\in A_p$.
\end{thm}
\begin{proof} For $f\in L^\vc$, we claim that
$$
\f{1}{w(B)}\int_B |(I- \mathcal{A}_{t_B})T(fw)(x)|\dx \leq C\|f\|_{L^\vc}
$$
for any ball $B\subset X$.

Set $f=f_1+f_2$ where $f_1=f\chi_{cB}$ with $c=\max\{c_1,4\}$. We
have
\begin{equation*}
\begin{aligned}
\f{1}{w(B)}\int_B |(I-\mathcal{A}_{t_B})T(fw)(x)|\dx&\leq \f{1}{w(B)}\int_B
|(I-\mathcal{A}_{t_B})T(f_1w)(x)|\dx\\
&~~~+\f{1}{w(B)}\int_B |(I-\mathcal{A}_{t_B})T(f_2w)(x)|\dx\\
&=I_1+I_2.
\end{aligned}
\end{equation*}
Let us estimate $I_1$ first. Since $w\in A_1$ then there exists
$r>1$ such that $w\in RH_r$. Using the $L^p$ boundedness of $T$ and
the Hardy-Littlewood maximal function, we have
\begin{equation*}
\begin{aligned}
I_1&\leq c\f{1}{w(B)}\int_B M(T(f_1w))(x)\dx\\
&\leq c\f{1}{w(B)}\|T(f_1w)\|_{L^r}V(B)^{1/r'}\\
&\leq c\|f\|_{L^\vc}\f{1}{w(B)}\Big(\int_{cB}w^r(x)\dx\Big)^{1/r}V(B)^{1/r'}\\
&\leq
c\|f\|_{L^\vc}\f{1}{w(B)}\f{w(B)}{V(B)}V(B)^{1/r}V(B)^{1/r'}=c\|f\|_{L^\vc}.
\end{aligned}
\end{equation*}
For the second term, by (b) we have
\begin{equation*}
\begin{aligned}
I_2&\leq \f{1}{w(B)}\int_B \int_{X\backslash
cB}K^1_{t_B}(x,y)(f_2w)(y)|\dy\dx\\
&\leq \f{1}{w(B)}\int_B \int_{X\backslash
cB}\f{1}{V(x,d(x,y))}\f{r_B^\alpha}{d(x,y)^\alpha}(f_2w)(y)|\dy\dx\\
&\leq c\|f\|_{L^\vc}\f{1}{w(B)}\int_B \int_{X\backslash
cB}\f{1}{V(x,d(x,y))}\f{r_B^\alpha}{d(x,y)^\alpha}w(y)|\dy\dx.
\end{aligned}
\end{equation*}
Since $c>4$, we have
\begin{equation*}
\begin{aligned}
I_2&\leq c\|f\|_{L^\vc}\sum_{j\geq 2}\f{1}{w(B)}\int_B
\int_{S_j(B)}\f{1}{V(x,d(x,y))}\f{r_B^\alpha}{d(x,y)^\alpha}w(y)|\dy\dx\\
&\leq c\|f\|_{L^\vc}\sum_{j\geq
2}2^{-j\alpha}\f{V(B)}{w(B)}\int_B\f{w(2^jB)}{V(2^jB)}\\
&\leq c\|f\|_{L^\vc}.
\end{aligned}
\end{equation*}
The boundedness of $T^*$ can be treated similarly. This completes our proof.
\end{proof}

\begin{thm}\label{thm-app2}
Let $T$ be an operator satisfying (H2). Then for any $w\in
A_1\cap RH_{p'_0}$, $T$ and $T^*$ are bounded from $wL^\vc(X)$ to $\BMOA$ and from
$wL^\vc(X)$ to ${\rm BMO}_{\mathcal{A}^*}(X)$. Then, by interpolation, $T$ is
bounded on $L^p_w(X)$ for all $p\in (p_0',\vc)$ and $w\in A_{p/p_0'}$.
\end{thm}
\begin{proof} For $f\in L^\vc$ and $w\in A_1\cap RH_{p'_0}$, we will claim that
$$
\f{1}{w(B)}\int_B |(I-A_{t_B})T(fw)(x)|\dx \leq C\|f\|_{L^\vc}
$$
for balls $B\subset X$.

Using the decomposition $f=\sum_{j\geq 2}f_j +f_0$ where $f_0=f\chi_{2B}$ and $f_j=f\chi_{S_j(B)}$, We
have
\begin{equation*}
\begin{aligned}
\f{1}{w(B)}\int_B |(I-\mathcal{A}_{t_B})T(fw)(x)|\dx&\leq \f{1}{w(B)}\int_B
|(I-\mathcal{A}_{t_B})T(f_0w)(x)|\dx\\
&~~~+\sum_{j\geq 2}\f{1}{w(B)}\int_B |(I-\mathcal{A}_{t_B})T(f_jw)(x)|\dx\\
&=I_0+\sum_{j\geq 2}I_j.
\end{aligned}
\end{equation*}
Since $w\in RH_{p_0'}$, using the $L^p$ boundedness of $T$ and
the Hardy-Littlewood maximal function, we have
\begin{equation*}
\begin{aligned}
I_0&\lesi \f{1}{w(B)}\int_B M(T(f_0w))(x)\dx\\
&\lesi \f{1}{w(B)}\|T(f_1w)\|_{L^{p_0'}}V(B)^{1/p_0}\\
&\lesi \|f\|_{L^\vc}\f{1}{w(B)}\Big(\int_{2B}w^{p_0'}(x)\dx\Big)^{1/p_0'}V(B)^{1/p_0}\\
&\lesi \|f\|_{L^\vc}\f{1}{w(B)}\f{w(B)}{V(B)}V(B)^{1/p_0}V(B)^{1/p_0'}=c\|f\|_{L^\vc}.
\end{aligned}
\end{equation*}
For $j\geq 2$, by (H2) and H\"older's inequality, we have
\begin{equation*}
\begin{aligned}
I_j&\leq \f{1}{w(B)}\int_B \int_{S_j(B)}|K^1_{t_B}(x,y)(f_jw)(y)|\dy\dx\\
&\leq \f{1}{w(B)}\int_B \Big(\int_{S_j(B)}|K^1_{t_B}(x,y)|^{p_0}\dy\Big)^{1/p_0}\Big(\int_{S_j(B)}|f_j(y)w(y)|^{p'_0}\dy\Big)^{1/p'_0}\dx\\
&\lesi \f{V(B)}{w(B)}2^{-j\delta}V(2^jB)^{1/p_0-1}\|f\|_{L^\vc}\Big(\int_{2^jB)}|w(y)|^{p'_0}\dy\Big)^{1/p'_0}\\
&\lesi \f{V(B)}{w(B)}2^{-j\delta}V(2^jB)^{1/p_0-1}\|f\|_{L^\vc}\f{w(2^jB)}{V(2^jB)}V(2^jB)^{1/p_0'}\\
&\lesi 2^{-j\delta}\f{V(B)}{w(B)}\f{w(2^jB)}{V(2^jB)}\|f\|_{L^\vc}.
\end{aligned}
\end{equation*}
Since $w\in A_1$,
$$
\f{V(B)}{w(B)}\f{w(2^jB)}{V(2^jB)}\leq C.
$$
Therefore,
$$
\sum_{j\geq 2}I_j\lesi \sum_j 2^{-j\delta}\|f\|_{L^\vc}\lesi \|f\|_{L^\vc}
$$
provided $\delta>0$.

This yields that $T$ is bounded from $wL^\vc(X)$ to $\BMOA$. The boundedness of $T^*$ can be treated similarly. This completes our proof.
\end{proof}
\subsection{Holomorphic functional calculi}
We now give some preliminary definitions of holomorphic functional
calculi as introduced by A. McIntosh \cite{Mc}.\\
Let $0\leq \nu<\pi$. We define the closed sector in the
complex plane $\mathbb{C}$
$$
S_\nu=\{z\in \mathbb{C}: |\arg z|\leq \nu\}
$$
and denote the interior of $S_\nu$ by $S^0_\nu $.\\

Let $H(S^0_\nu)$ be the space of all
holomorphic functions on $S_\nu^0$.
We define the following subspaces of  $H(S^0_\nu)$:
$$
H_\vc(S_\nu^0)=\{b\in H(S^0_\nu): ||b||_\vc<\vc\},
$$
where $||b||_{\vc}=\sup\{|b(z)|: z\in S_\nu^0\}$, and
$$
\Psi(S^0_\nu)=\{\psi\in H(S^0_\nu): \exists s>0, |\psi(z)|\leq
c|z|^s(1+|z|^{2s+1})^{-1}\}.
$$

Let $L$ be a linear
operator of type $\omega$ on $L^2(X)$ with $\omega < \pi/2$; hence
$L$ generates a holomorphic semigroup $e^{-zL}, 0 \leq |Arg(z)| <
\pi/2 - \omega$. Assume the following two conditions.\\
\textbf{Assumption (a):} The holomorphic semigroup $e^{-zL}, 0 \leq
|Arg(z)| < \pi/2 - \omega$, is represented by the kernel $p_z(x, y)$
which satisfies the Gaussian upper bound
\begin{equation}\label{GaussianUpperBound}
|p_z(x,y)|\leq
c_{\theta}\f{1}{V(x,|z|^{1/m})}\exp\Big(-\f{d(x,y)^{m/(m-1)}}{c|z|^{1/(m-1)}}\Big)
\end{equation}
for $x,y \in \RR, |Arg(z)| <
\pi/2 - \theta$ for $\theta> \omega$.

\noindent \textbf{Assumption (b):} The operator $L$ has a  bounded
$H_\vc$-calculus on $L^2(X)$. That is, there exists $c_{\nu,2} >
0$ such that $b(L) \in \mathcal{L}(L^2, L^2)$, and for $b \in
H_\vc(S^0_\nu)$,
$$
||b(L)f||_2\leq c_{\nu,2}||b||_\vc||f||_{2}
$$
for any $f\in L^2(X)$.\\

We have the following result.
\begin{thm}\label{thm-fc}
Let $L$ satisfy conditions (a) and (b) and let $f\in H_\vc(S^0_\nu)$. Then for any $w\in
A_1$, $f(L)$ and $[f(L)]^*$ are bounded from $wL^\vc(X)$ to $\BMOA$ and from
$wL^\vc(X)$ to ${\rm BMO}_{\mathcal{A}^*}(X)$ where $\mathcal{A}_t=e^{-tL}$. Then, by interpolation, $g(L)$ is
bounded on $L^p_w(X)$ for all $p\in (1,\vc)$ and $w\in A_p$.
\end{thm}
Note that in the similar condition, it was proved in \cite{DM} that the functional calculus $f(L)$ is of weak type $(1,1)$ and hence bounded on $L^p(X)$ for all $1< p<\vc$. Moreover, the weighted estimates for the functional calculus $f(L)$ were investigated in \cite{Ma} in which the author proved that $f(L)$ is bounded on $L^p_w(X)$ for all $1<p<\vc$ and $w\in A_p$. Here, in Theorem \ref{thm-fc}, we prove the weighted endpoint estimates for the functional calculus $f(L)$ and then by the interpolation theorem we regain the weighted estimates for $f(L)$.
\begin{proof}
By the convergence lemma in \cite{Mc}, we can assume that $f\in \Psi(S^0_\nu)$. Then, it was proved in \cite{DM} that $g(L)$ and $[g(L)]^*$ satisfy \textbf{(H1)} with $\mathcal{A}_t=e^{-tL}$. Hence, the desired result follows directly from Theorem \ref{thm-appl1}.
\end{proof}
\subsection{Spectral multipliers}
Let $L$ be a non-negative self-adjoint
operator on $L^2({X})$ and the operator $L$ generates an analytic
semigroup $\{e^{-tL}\}_{t>0}$ whose kernels $p_t(x,y)$ satisfies Gaussian
upper bound:
\begin{equation}\label{GE}
|p_t(x,y)|\leq
\f{C}{V(x,t^{1/m})}\exp\Big(-c\f{d(x,y)^{m/(m-1)}}{t^{1/(m-1)}}\Big)
\end{equation}
for all $x, y\in X$ and $t>0$.

By the spectral theorem, for any bounded Borel function
$F: [0, \infty)\rightarrow {\Bbb C}$, one can define the operator
\begin{equation}\label{eq1-defspectralmultiplier}
F(L)=\int_0^{\infty} F(\lambda) dE(\lambda)
\end{equation}
 which is  bounded on $L^2(X)$. We have the following result.
\begin{thm}\label{thm-spectralmultiplier}
Let $L$ be a non-negative self-adjoint
operator  satisfying (GE). Suppose that $n>s>n/2$ and for any $R>0$ and all Borel
functions $F$ such that {\rm supp}$F\subset [0,R]$,
\begin{equation}\label{eq3.1}
\int_X |K_{F(\sqrt[m]{L})}(x,y)|^2\dx\leq
\f{C}{V(y,R^{-1})}\|\delta_RF\|^2_{L^q}
\end{equation}
for some $q\in [2,\infty]$. Then for any Borel function $F$ such
that $$\sup_{t>0}\|\eta \delta_tF\|_{W^q_s}<\infty,$$
where $ \delta_t F(\lambda)=F(t\lambda)$,
   $\| F \|_{W^q_s}=\|(I-d^2/d x^2)^{s/2}F\|_{L^q}$,
the operator
$F(L)$ and $F(L)^*=\overline{F}(L)$ is bounded from $wL^\vc$ to $\BMOA$ for all $w\in A_1\cap RH_{r_0'}$, where $\mathcal{A}_t=I-(I-e^{-tL})^M$ for $M>\f{s}{m}$ and $r_0=n/s$. Hence by interpolation,
 $F(L)$ is bounded on $L^p_w(X)$ for $w\in A_{p/r_0}$ and $p\in (r_0,\vc)$.
\end{thm}
Note that under the condition as in Theorem \ref{thm-spectralmultiplier}, it was prove in \cite{DOS} that the spectral multiplier $F(L)$ is of weak type $(1,1)$ and hence bounded on $L^p_w(X)$. The weighted estimates for $F(L)$ was studied in \cite{AD1, DSY}. The main contribution in Theorem \ref{thm-spectralmultiplier} is the weighted endpoint estimates for the spectral multipliers $F(L)$.

\begin{proof} From the proof of Theorem 4.5 in \cite{AD1}, we get that \textbf{(H2)} holds for $T:=F(L)$ and the family $\mathcal{A}_t:=I-(I-e^{-tL})^M$ for $M>\f{s}{m}$ and all $p_0<r'_0$. Hence, using Theorem \ref{thm-app2} and letting $p_0\to r_0'$ , we get the desired result.
\end{proof}

\end{document}